\documentclass{amsart}
\usepackage{amssymb,enumerate}

\usepackage[all,2cell,ps]{xy}
\usepackage{hyperref}

\newtheorem{theorem}[subsection]{Theorem}

\newtheorem{lemma}[subsection]{Lemma}

\theoremstyle{definition}

\newtheorem{example}[subsection]{Example}
\newtheorem{conjecture}[subsection]{Conjecture}
\newtheorem{question}[subsection]{Question}

\theoremstyle{remark}

\numberwithin{equation}{subsection}

\newcommand{\fm}{\mathfrak{m}}
\newcommand{\fn}{\mathfrak{n}}

\newcommand{\ff}{\mathfrak{f}}

\newcommand{\tf}{\boldsymbol{\bot}}
\newcommand{\tp}{\boldsymbol{\top}}

\newcommand{\ov}{\overline}

\newcommand{\QQ}{\mathbb{Q}}

\newcommand{\End}{\operatorname{End}}
\newcommand{\Ext}{\operatorname{Ext}}

\newcommand{\Hom}{\operatorname{Hom}}

\begin{document}

\title[Torsion in tensor products]
{Torsion in tensor products over one-dimensional domains}
\author[Steinburg, Wiegand]
{Neil Steinburg and Roger Wiegand}

\address{Neil Steinburg\\
Department of Mathematics and Computer Science \\
Drake University\\
2507 University Avenue\\
Des Moines IA 50311, U.S.A.}
\email{neil.steinburg@drake.edu }

\address{Roger Wiegand \\
Department of Mathematics \\
University of Nebraska\\
Lincoln, NE 68588-0130, U.S.A.}
\email{rwiegand1@unl.edu}

\thanks{Wiegand's research was partially supported by Simons Collaboration Grant 426885}



\date{\today}

\begin{abstract}
Over a one-dimensional Gorenstein local domain $R$, let $E$ be the endomorphism ring of the maximal of $R$, viewed as a subring of the integral closure $\ov R$.  If there exist finitely generated $R$-modules $M$ and $N$, neither of them free, whose tensor product is torsion-free, we show that $E$ must be local with the same residue field as $R$.  
\end{abstract}

\maketitle{}

\section{Introduction}

Finding interesting examples of non-zero, finitely generated modules $M, N$ over a commutative Noetherian ring $R$, with $M\otimes_RN$ torsion-free (meaning that no non-zero element of $M\otimes_RN$ is killed by a regular element of $R$) is a non-trivial task.  Of course there are boring examples: take one of the modules to be torsion-free and the other to be projective.  Or, if $R$ is not local, take $M = R/\fm$ and $N = R/\fn$, where $\fm$ and $\fn$ are distinct maximal ideals.  A slightly less boring example is obtained by taking $R=\QQ[[x,y]]/(xy)$ and $M = N = R/(x)$.

\begin{question}\label{question:exist?}  Let $R$ be a local domain, and let $M$ and $N$ be finitely generated modules, neither one of them free.  Must $M\otimes_RN$ always have non-zero torsion?
\end{question}

Again, the answer is ``no'', and here is the connection with numerical semigroups:

\begin{example} Let $R = k[[t^4,t^5,t^6]]$, $M = (t^4,t^5)$, and $N = (t^4,t^6)$.  Then $M\otimes_RN$ is torsion-free \cite[4.3]{HW}. 
\end{example}

In fact, the only known examples where Question~\ref{question:exist?} has a negative answer are numerical semigroup rings.  This leads to a (somewhat halfhearted, since it is probably false) conjecture:

\begin{conjecture}\label{conjecture:halfhearted}
Suppose $R$ is a one-dimensional local domain whose integral closure $\ov R$ is finitely generated as an $R$-module.  If there exist finitely generated modules $M$ and $N$, neither of them free, with $M\otimes_RN$ torsion-free, then $\ov R$ is local, and the inclusion $R\subseteq \ov R$ induces an isomorphism on residue fields.
\end{conjecture}

\section{Some evidence}  In this section we will prove the result stated in the abstract, which gives some support (admittedly rather sketchy) for Conjecture~\ref{conjecture:halfhearted}.

Throughout, $(R,\fm,k)$ is a one-dimensional Gorenstein local domain, with maximal ideal $\fm$ and residue field $k=R/\fm$.  We let $K$ denote the quotient field of $R$. If $I$ and $J$ are non-zero $R$-submodules of $K$, we identify $\Hom_R(I,J)$ with the set $\{\alpha\in K \mid \alpha I \subseteq J\}$, via the isomorphism $\varphi \mapsto \frac{1}{a}\varphi(a)$, where $a$ is a fixed but arbitrary nonzero element of $I$. In particular, we identify $\End_R\fm$ with the ring $E = \{\alpha\in K \mid \alpha \fm\subseteq \fm\}$. Then $R\subseteq E \subseteq \ov R$, where $\ov R$ is the integral closure of $R$ in $K$.  The next lemma is due to Bass \cite{Bass}.

\begin{lemma}\label{lemma:simple} Assume $\fm$ is not a principal ideal.  Then $E/R$ is a simple $R$-module, and $E$ is minimally generated, as an $R$-module, by $\{1,y\}$, where $y$ is an arbitrary element of $E\setminus R$.
\end{lemma}
\begin{proof} Since $\fm$ is indecomposable, there is no surjection $\fm \twoheadrightarrow R$.  (Such a surjection would split, giving a decomposition $\fm \cong R\oplus H$, with $H \ne 0$, as $\fm$ is not prinicipal; but clearly $\fm$ is indecomposable, since $R$ is a domain.) This gives the second equality in the display
\begin{equation}\label{eq:dual=End}
\fm^* = \Hom_R(\fm,R) = \Hom_R(\fm,\fm) = E\,.
\end{equation}
Dualizing the short exact sequence 
\[
0 \to \fm \to R \to k \to 0\,,
\]
and using the fact that $k^* = 0$, we get an exact sequence
\[
0\to R^* \to \fm^* \to \Ext_R^1(k,R)\to 0\,.
\]
But $\Ext^1_R(k,R) \cong k$, as $R$ is one-dimensional and Gorenstein.  The identification of $\End_R(\fm)$ with $E$ is compatible with the identification of $R^*$ with $R$ (via multiplications), and thus the last short exact sequence shows that  $E/R\cong k$.  The next assertion is clear from simplicity of $E/R$ and the fact that $1$ is part of a minimal generating set for $E$, as $1\notin \fm = \fm E$.
\end{proof}

\begin{lemma} \label{lemma:iso}
Let $S$ be a subring of $K$ containing $R$ and finitely generated as an $R$-module. Let $M$ and $N$ be 
finitely generated $S$-modules such that $M\otimes_RN$ is torsion over $R$.  Then the natural surjection $M\otimes_R N \twoheadrightarrow M\otimes_S N$ is an isomorphism. 
\end{lemma}

\begin{proof}
We consult the following commutative diagram:

\begin{equation}
\label{eq:CD}
\xymatrix{ 
M\otimes_RN\  \ar@{>->}[r]^{\delta\ \ \ \ \ } \ar@{->>}[d]^{\alpha} & K\otimes_R(M\otimes_RN) 
\ar@{->}[r]^{\cong \ \ \ \ \ \ } \ar@{->>}[d]^{\beta} & (K\otimes_RM)\otimes_K(K\otimes_RN) \ar@{->>}[d]^{\gamma} \\
M\otimes_SN \ar@{->}[r]^{\varepsilon\ \ \ \ \ } & K\otimes_S(M\otimes_SN) 
		\ar@{->}[r]^{\cong \ \ \ \ \ \ } & (K\otimes_SM)\otimes_K(K\otimes_SN) } 
\end{equation}

The map $\delta$ is injective because $M\otimes_RN$ is torsion-free.  One checks (by clearing denominators) that a subset of an $S$-module is linearly independent over $S$ if and only if it is linearly independent over $R$, and so its rank as an $S$-module equals its rank as an $R$-module.  
Thus $r:= \dim_K(K\otimes_RM) = \dim_K(K\otimes_SM)$ and $s:=\dim_K(K\otimes_RN) = \dim_K(K\otimes_SN)$.
The surjective map $\gamma$ is therefore an isomorphism, since its domain 
and target both have the same $K$-dimension, namely $rs$.  From the diagram, we see that $\beta$ must be an isomorphism too, and hence $\alpha$ is injective. 
\end{proof}

\begin{theorem}\label{theorem:evidence}
Let $(R,\fm,k)$ be a Gorenstein local domain of dimension one, and let $E=\End_R(\fm)$, viewed as a ring between $R$ and its integral closure $\overline R$.  Assume that there exist finitely generated modules $M$ and $N$, neither of them free, such that $M\otimes_RN$ is torsion-free.  Then $E$ is local, and the inclusion $R\to E$ induces a bijection on residue fields.
\end{theorem}
\begin{proof} If $\fm$ is a principal ideal, then $R$ is a discrete valuation ring, and $R=E=\overline R$.  Therefore we assume from now on that $\fm$ is not principal.

We begin with some reductions. We first get rid of free summands, by writing $M=M'\oplus R^m$ and $N = N'\oplus R^n$, where
both $M'$ and $N'$ are non-zero, and neither has a non-zero free direct summand.  Notice that $M'\otimes_RN'$, being a direct summand of $M\otimes_RN$, is torsion-free.  Replacing $M$ by $M'$ and $N$ by $N'$, we may assume that neither $M$ nor $N$ has a non-zero free direct summand. 

Next, we have a reduction that goes back to Auslander's 1961 paper \cite{Au}. Let $\tp X$ denote the torsion 
submodule of a module $X$, and put $\tf X= X/(\tp X)$.  By 
\cite[Lemma 2.2]{CW}, $(\tf M)\otimes_R(\tf M)$ is torsion-free.  Moreover, both $\tf M$ and $\tf N$ are non-zero, since otherwise $M\otimes_RN$ would be a non-zero torsion module.   We claim that $\tf M$ has no non-zero free summand.  For, suppose there is a surjection $\tf M \twoheadrightarrow R$. Composing this with the natural surjection $M \twoheadrightarrow \tf M$, we get a surjection $M \twoheadrightarrow R$, and hence $M\cong R \oplus L$, a contradiction. Similarly, $\tf N$ has no non-zero free summand. Replacing $M$ and $N$ by their reductions modulo torsion, we may assume that both $M$ and $N$ are non-zero torsion-free $R$-modules, and that neither $M$ nor $N$ has a non-zero free direct summand.

As in \cite{Bass}, we note that every homomorphism $M\to R$ has its image in $\fm$, and so $M^* = \Hom_R(M,\fm)$, which has a natural $E$-module structure extending the $R$-module structure.  Therefore $M^{**}$ is also an $E$-module.  Since $R$ is Gorenstein and $M$ is torsion-free (= maximal Cohen-Macaulay), the natural map $M\to M^{**}$ is an isomorphism, and hence $M$ itself has an $E$-module structure compatible with the original $R$-module structure.  By symmetry, $N$ too has a compatible $E$-module structure. Lemma \ref{lemma:iso} shows that the natural surjection $M\otimes_R N \twoheadrightarrow M\otimes_E N$ is an isomorphism and, in particular, $M \otimes_EN$ is torsion-free.

Suppose, by way of contradiction, that $E$ is not local, and put $A = E/\fm E$. 
This is a $2$-dimensional $k$-algebra, and it is not local and hence must be isomorphic to $k\times k$.  Let $e$ be the idempotent of $A$ supported on first coordinate.  Then neither $e$ nor $1-e$ is a unit of $A$.  Let $\ov M 
= M/\fm M$ and $\ov N = N/\fm N$.  We claim that $e\ov M \ne 0$.  For suppose $e\ov M = 0$.  Lift $e$ to an element $\tilde e\in E$.  Then $\tilde e M \subseteq \fm M$.  Moreover, $\tilde e M + (1-\tilde e) M + \fm M = M$, and hence $(1-\tilde e)M= M$ by Nakayama's Lemma. 
The Determinant Trick yields an element $a\in (1-\tilde e)E$ such that 
$(1+a)M=0$.  But $M$ is faithful as an $R$-module and hence as an $E$-module (clear denominators).  
Therefore $1+a=0$, and hence $-1\in (1-\tilde e)E$.  But then $-1\in (1-e)A$, 
contradicting the fact that $1-e$ is not a unit.  This proves the claim and shows that $e\ov M \ne 0$.  By symmetry, $(1-e) \ov N \ne 0$, and hence 
$e\ov M\otimes_k (1-e)\ov N \ne 0$.  However, the isomorphism 
$\alpha: M\otimes_RN\overset \cong \longrightarrow  M\otimes_EN$ induces an isomorphism 
$\ov M\otimes_k\ov N \overset \cong \longrightarrow \ov M\otimes_A\ov N$, carrying the non-zero module
$e\ov M\otimes_k (1-e)\ov N$ onto $e\ov M\otimes_A (1-e)\ov N = 0$, a contradiction.  
This completes the proof that $E$ is local.

Let $\fn$ be the maximal ideal of $E$, and put $\ell = E/\fn$.  Suppose $\dim_k\ell > 1$.  
The inclusion $\fm E \hookrightarrow \fn$ induces a surjection $E/\fm E \twoheadrightarrow E/\fn = \ell$.
Since, by Lemma~\ref{lemma:simple}, $\dim_k(E/\fm E) = 2$, this surjection must be an isomorphism, and
hence $\fn = \fm E = \fm$.   Observe that the isomorphism $\alpha:M\otimes_RN\to M\otimes_EN$ induces
an isomorphism
\begin{equation}
\label{eq:iso}
\ov M\otimes_k \ov N \overset {\cong} {\longrightarrow} \ov M\otimes_{\ell} \ov N \,.
\end{equation}

 Put $u = \dim_\ell \ov M$ and $v=\dim_\ell\ov N$.  Then $\dim_\ell(\ov M\otimes_\ell \ov N) = uv$, and hence
 $\dim_k(\ov M\otimes_\ell \ov N) = 2uv$.  On the other hand, $\dim_k(\ov M\otimes_k\ov N) = (\dim_k\ov M)(\dim_k\ov N) = (2u)(2v) = 4uv$.  The isomorphism in \eqref{eq:iso} forces $4uv = 2uv$, and hence either $u=0$ or $v=0$, contradicting Nakayama's Lemma.  This shows that $\dim_k\ell = 1$, and the proof is complete. 
\end{proof}

One might hope, at least for a Gorenstein ring $(R,\fm,k)$ with finite integral closure $\ov R$, 
that $E$ being local with residue field $k$ would force $\ov R$ to be local with residue field $k$.  Of course,
Theorem~\ref{theorem:evidence} would then answer Conjecture~\ref{conjecture:halfhearted} affirmatively.  The next example dashes this hope.

\begin{example}
\label{example:hopedash}
Let $k$ be a field and $D=k[X]_{(X)\cup (X-1)}$.  Then $D$ is a principal ideal domain with 2 maximal ideals. Let $A = k[T]/(T^2), B = k[X]/(X^2) \times k[X]/(X-1)^2$, and define $i:A \hookrightarrow B$ by $i(a+bt) = (a+bx,a+b(x-1))$ where $a,b \in k$, and decapitalization of the indeterminates indicates passage to cosets. Let $\pi: D \twoheadrightarrow B$ be the composition of the natural projection of $D \twoheadrightarrow D/(X^2(X-1)^2)$ and the isomorphism $D/(X^2(X-1)^2) \overset {\cong}{\longrightarrow} B$ provided by the Chinese Remainder Theorem. Define $R$ to be the pullback of $i$ and $\pi$:
\begin{equation}
%\label{eq:smallCD}
\xymatrix{ 
R\ \ \ar@{>->}[r] \ar@{->>}[d] & D \ar@{->>}[d]^{\pi\ \ } \\
A\ \  \ar@{>->}[r]^i  & B } 
\end{equation}
By \cite[Proposition 3.1]{WW}, $(R,\fm,k)$ is a local one-dimensional domain, $\ov R = D$, and $\ov R$ is
 finitely generated as an $R$-module. Furthermore, letting $\ff$ be the conductor, we have 
 $A \cong R/\ff$ and $B \cong D/\ff$. Since the length of $\ov R/\ff$, namely $4$, is twice the length of $R/\ff$, 
 \cite[Corollary 6.5]{Bass} guarantees that $R$ is Gorenstein.  One checks that $E:= \End_R(\fm)$ is local, with residue field $k$, but $\ov R$ is not local.
 \end{example}
 
This example cannot be promoted to a counterexample to Conjecture~\ref{conjecture:halfhearted}.  To see this, first observe that $B$ is generated by two elements as an $A$-module.  It follows that $\ov R = D$ is two-generated as an $R$-module.  Therefore $R$ has multiplicity two \cite[Theorem 2.1]{G}, and hence every ideal of the completion $\widehat R$ is two-generated.  It follows that $\widehat R$ is a hypersurface and therefore, by the main theorem of  \cite{HW}, the tensor product of any two non-free finitely generated $R$-modules has non-zero torsion.

\medskip

Some of the material in this paper is taken from the first-named author's 2018 Ph.D. dissertation at the University of Nebraska.

The authors thank the anonymous referee for several helpful suggestions, which have significantly improved the exposition.

\bibliographystyle{plain}

\end{document}